\documentclass[psamsfonts]{amsart}

 \makeatletter
       \def\@makefnmark{%
               \leavevmode
               \raise.9ex\hbox{\check@mathfonts
                       \fontsize\sf@size\z@\normalfont%
                               \@thefnmark}%
       }
       \makeatother
       
\usepackage{amsmath,amssymb,mathrsfs}
\usepackage[dvipdfmx,usenames]{color}
\usepackage[all]{xy}
\usepackage[dvipdfmx]{graphicx}
\usepackage{wrapfig}
\usepackage{tabularx}
\usepackage{supertabular}
\usepackage{setspace}
\usepackage{xypic}
\usepackage{enumerate} 

\newtheorem{dfn}{Definition}[section]
\newtheorem{thm}[dfn]{Theorem}
\newtheorem{lem}[dfn]{Lemma}

\newtheorem{rmk}[dfn]{Remark}

\newcommand{\ind}{\mathrm{index}}

\title{A note on Atiyah's $\Gamma$-index theorem in Heisenberg calculus}
\author{Tatsuki SETO}

\address{Graduate School of Mathematics, Nagoya University, Furocho, Chikusaku, Nagoya, Japan}
\email{m11034y@math.nagoya-u.ac.jp}
\subjclass[2000]{Primary 19K56; Secondary 46L87.}
\keywords{Index theory, Heisenberg calculus, Galois covering, 
	Heisenberg structure, contact structure}

\begin{document}

\begin{abstract}
In this note, 
we prove an index theorem on Galois coverings 
for Heisenberg elliptic differential operators, 
but not elliptic, 
which is analogous to Atiyah's $\Gamma$-index theorem. 
This note also contains 
an example of Heisenberg differential operators with a non-trivial $\Gamma$-index. 
\end{abstract}

\maketitle

\section*{Introduction}

M. F. Atiyah \cite{MR0420729} 
introduced the notion of the $\Gamma$-index $\ind_{\Gamma}(\widetilde{D})$ 
for a lifted elliptic differential operator $\widetilde{D}$ 
on a Galois $\Gamma$-covering over a closed manifold 
and proved that the $\Gamma$-index $\ind_{\Gamma}(\widetilde{D})$ 
of a lifted operator $\widetilde{D}$ 
equals the Fredholm index $\ind (D)$ of the elliptic differential operator $D$ 
on the base manifold. 
On the other hand, Atiyah \cite{MR0420729} also 
investigate properties of a $\Gamma$-trace $\mathrm{tr}_{\Gamma}$  at the same time. 
The $\Gamma$-trace is a trace of the $\Gamma$-trace operators, 
so it induces a homomorphism $(\mathrm{tr}_{\Gamma})_{\ast}$ 
from 
$K_{0}$-group of the $\Gamma$-compact operators to the real numbers. 
Out of a lifted elliptic differential operator $\widetilde{D}$, 
we can define the $\Gamma$-index class $\mathrm{Ind}_{\Gamma}(\widetilde{D})$ 
by using the Connes-Skandalis idempotent \cite[II.9.$\alpha$ (p.131)]{MR1303779} 
and send it by the induced homomorphism $(\mathrm{tr}_{\Gamma})_{\ast}$, 
then the image $(\mathrm{tr}_{\Gamma})_{\ast}(\mathrm{Ind}_{\Gamma}(\widetilde{D}))$ 
equals the $\Gamma$-index $\ind_{\Gamma}(\widetilde{D})$ 
and thus the Fredholm index $\ind (D)$ of 
the elliptic differential operator $D$ on the base manifold. 

On the other hand, 
there is another pseudo-differential calculus 
on Heisenberg manifolds which is called 
the Heisenberg calculus; see, for instance \cite{MR2417549}. 
Roughly speaking, Heisenberg calculus is 
``weighted'' calculus and the 
product of the ``Heisenberg principal symbols'' are defined by convolution product. 
When the Heisenberg principal symbol of $P$ 
is invertible, we call $P$ a Heisenberg elliptic operator. 
Note that any Heisenberg elliptic operator is not elliptic. 
For a Heisenberg elliptic operator $P$, 
we can construct a parametrix by using its inverse, 
so $P$ is a Fredholm operator if the base manifold is closed. 
Thus the Fredholm index of $P$ on a closed manifold 
is well defined, but 
a solution of an index problem of $P$ 
does not obtained in general. 
However, 
index problems for Heisenberg elliptic operators 
on contact manifolds or foliated manifolds 
are solved by E. van Erp and P. F. Baum; 
see \cite{MR3261009}, \cite{MR2680395}, \cite{MR2680396}, \cite{MR2746652}. 

In this note, 
we study that we can define 
the $\Gamma$-index $\ind_{\Gamma}(\widetilde{P})$ 
and the $\Gamma$-index class $\mathrm{Ind}_{\Gamma}(\widetilde{P})$ 
for a lifted Heisenberg elliptic differential operator $\widetilde{P}$ 
by using a parametrix. 
Once these ingredients are defined, 
the proof of the matching of three ingredients 
$\ind_{\Gamma}(\widetilde{P})$, 
$(\mathrm{tr}_{\Gamma})_{\ast}(\mathrm{Ind}_{\Gamma}(\widetilde{P}))$ 
and $\ind (P)$, 
is straightforward; see subsection \ref{sub:main}.  
We also investigate an example 
of Heisenberg differential operators 
on a contact manifold
with non-trivial $\Gamma$-index 
by using the index formula in \cite{MR3261009}; see subsection \ref{sub:exm}. 


\section{Short review of Atiyah's $\Gamma$-index theorem}
\label{sec:atiyah}

In this section, 
we recall Atiyah's $\Gamma$-index theorem in 
ordinary pseudo-differential calculus. 
The main reference of this section is 
Atiyah's paper \cite{MR0420729}. 
Let $\widetilde{M} \to M$ be a Galois covering 
with a deck transformation group $\Gamma$ 
over a closed manifold $M$ with a smooth measure $\mu$
and $D : C^{\infty}(E) \to C^{\infty}(F)$ 
an elliptic differential operator 
on Hermitian vector bundles $E,F \to M$. 
We lift these ingredients on $\widetilde{M}$ 
and denote by 
$\widetilde{D} : C^{\infty}(\widetilde{E}) \to C^{\infty}(\widetilde{F})$ 
and $\tilde{\mu}$. 
Let $\mathrm{Ker}_{L^{2}}(\widetilde{D})$ (resp. $\mathrm{Ker}_{L^{2}}(\widetilde{D}^{\ast})$) 
be the  $L^{2}$-solutions of $\widetilde{D}u = 0$ (resp. $\widetilde{D}^{\ast}u = 0$) 
and 
denote by $\Pi_{0}$ (resp. $\Pi_{1}$) 
the orthogonal projection on a 
closed subspace 
$\mathrm{Ker}_{L^{2}}(\widetilde{D})$ (resp. $\mathrm{Ker}_{L^{2}}(\widetilde{D}^{\ast})$) 
of the $L^{2}$-sections. 

A $\Gamma$-invariant bounded operator $T$ on the 
$L^{2}$-sections $L^{2}(\widetilde{E})$ 
of $\widetilde{E}$ 
is of \textit{$\Gamma$-trace class} 
if $\phi T \psi \in L^{2}(\widetilde{E})$ 
is of trace class for any 
compactly suppoted smooth functions $\phi , \psi$ 
on $\widetilde{M}$.  
Denote by $\mathcal{L}^{1}_{\Gamma}$ 
the set of $\Gamma$-trace class operators 
and $\mathrm{tr}_{\Gamma}(T)$ the \textit{$\Gamma$-trace} 
of a $\Gamma$-trace class operator $T$ defined by 
\[
\mathrm{tr}_{\Gamma}(T) = \textrm{Tr}(\phi T \psi) \in \mathbb{C}. 
\]
Here, the right hand side is the trace of 
a trace class operator $\phi T \psi$ 
and this quantity does not depend on 
the choice of functions $\phi , \psi$. 
By using ellipticity of $\widetilde{D}$, 
operators $\phi \Pi_{0} \psi$ and $\phi \Pi_{1} \psi$ 
are smoothing operators on compact sets. 
Thus $\Pi_{0}$ and $\Pi_{1}$ are of $\Gamma$-trace class  
and thus one obtains the $\Gamma$-index of $\widetilde{D}$: 
\[
\ind_{\Gamma}(\widetilde{D}) 
	= \mathrm{tr}_{\Gamma}(\Pi_{0}) - \mathrm{tr}_{\Gamma}(\Pi_{0}) \in \mathbb{R}. 
\]

In the context of the $\Gamma$-index theorem, 
the most important class of $\Gamma$-trace class operators 
is the lifts of almost local smoothing operators on $M$. 
Let $S$ be an almost local smoothing operator with 
a smooth kernel $k_{S}$ 
and $\widetilde{S}$ a lift of $S$. 
Then $\widetilde{S}$ is of $\Gamma$-trace class 
and its $\Gamma$-trace is calculated by the following: 
\begin{equation}
\label{eq:gammatr}
\mathrm{tr}_{\Gamma}(\widetilde{S}) 
	= \int_{M}\mathrm{tr}(k_{S}(x,x)) d\mu 
	= \mathrm{Tr}(S).  \tag{$\ast$}
\end{equation}

Denote by $\mathcal{K}_{\Gamma}$ 
the $C^{\ast}$-closure of $\mathcal{L}^{1}_{\Gamma}$ 
and $K_{0}(\mathcal{K}_{\Gamma})$ 
the analytic $K_{0}$-group. 
Then $\mathrm{tr}_{\Gamma}$ induces 
a homomorphism of abelian groups by substitution: 
\[
(\mathrm{tr}_{\Gamma})_{\ast} 
	: K_{0}(\mathcal{K}_{\Gamma}) \to \mathbb{R}. 
\]

On the other hand, 
since $D$ is an elliptic differential operator, 
there exist an almost local parametrix $Q$ and 
almost local smoothing operators $S_{0},S_{1}$ 
such that one has 
$QD = 1 - S_{0}$ and $DQ = 1 - S_{1}$. 
Denote  by 
$\widetilde{Q}$, $\widetilde{S_{0}}$ and $\widetilde{S_{1}}$ 
lifts of these operators 
and then one has same relations 
$\widetilde{Q}\widetilde{D} = 1 - \widetilde{S_{0}}$ 
and $\widetilde{D}\widetilde{Q} = 1 - \widetilde{S_{1}}$. 
Set 
\[ 
e_{\widetilde{D}} = 
\begin{bmatrix}
\widetilde{S_{0}}^{2} & 
	\widetilde{S_{0}}(1 + \widetilde{S_{0}}) \widetilde{Q} \\
\widetilde{S_{1}}\widetilde{D} & 1-\widetilde{S_{1}}^{2}
\end{bmatrix} \quad \text{and} \quad 
e_{1} = 
\begin{bmatrix}
0 & 0 \\ 
0 & 1
\end{bmatrix}. 
\]
By $\widetilde{Q}\widetilde{S_{1}} = \widetilde{S_{0}}\widetilde{Q}$ 
and $\widetilde{S_{1}}\widetilde{D} = \widetilde{D}\widetilde{S_{0}}$, 
one has $e_{\widetilde{D}}^{2} = e_{\widetilde{D}}$, that is, 
$e_{\widetilde{D}}$ is an idempotent. 
Note that this idempotent $e_{\widetilde{D}}$ 
is called the Connes-Skandalis idempotent; 
see, for instance \cite[II.9.$\alpha$ (p.131)]{MR1303779}. 
Moreover, a difference $e_{\widetilde{D}} - e_{1}$ 
is of $\Gamma$-trace class. 
Hence 
we can define a $\Gamma$-index class 
\[
\mathrm{Ind}_{\Gamma}(\widetilde{D}) 
	= [e_{\widetilde{D}}] - [e_{1}] \in K_{0}(\mathcal{K}_{\Gamma}). 
\]
By the definition of a map $(\mathrm{tr}_{\Gamma})_{\ast}$ 
and Atiyah's paper, one has the following: 

\begin{thm}[Atiyah's $\Gamma$-index theorem {\cite[Theorem 3.8]{MR0420729}}]
In the above settings, we have the following equality: 
\[ 
\ind_{\Gamma}(\widetilde{D}) 
	= (\mathrm{tr}_{\Gamma})_{\ast}(\mathrm{Ind}_{\Gamma}(\widetilde{D}))
	= \ind (D) \in \mathbb{Z}. 
\]
\end{thm}

As described in subsection \ref{sub:main}, 
Atiyah's proof of matching of these ingredients in the above equality
does not essentially use ellipticity. 
Note that 
ellipticity of $D$ and $\widetilde{D}$ is only used in 
the definition of these ingredients. 

\section{Atiya's $\Gamma$-index theorem in Heisenberg calculus}

Let $(M,H)$ be a closed Heisenberg manifold, 
that is, $M$ is a closed manifold and $H \subset TM$ is a hyperplane bundle. 
Let 
$P : C^{\infty}(E) \to C^{\infty}(F)$ 
be a Heisenberg elliptic differential operator on 
Hermitian vector bundles $E,F \to (M,H)$, 
that is, 
the Heisenberg principal symbol $\sigma_{H}(P)$ of $P$ 
is an invertible element. 
In this section, 
we prove the $\Gamma$-index theorem for $P$, 
which is analogous to 
Atiyah's $\Gamma$-index theorem. 
Note that $P$ is not an elliptic operator 
in the sense of ordinary pseudo-differential calculus. 

\subsection{Statement and proof}
\label{sub:main}

By \cite[Proposition 3.3.1]{MR2417549}, 
there exist parametrix $Q$ and 
smoothing operators $S_{0},S_{1}$ 
such that one has 
$QP = 1 - S_{0}$ and $PQ = 1 - S_{1}$. 
Thus  
$P$ is a Fredholm operator 
and one has the Fredholm index $\ind (P) \in \mathbb{Z}$ of $P$
by compactness of $M$. 
Moreover, since a integral kernel of $Q$ 
is smooth off the diagonal, 
we can choose $Q$ as 
an almost local operator 
and then $S_{0}$ and $S_{1}$ are 
also almost local operators. 

Let $\widetilde{M} \to M$ be a Galois covering 
with a deck transformation group $\Gamma$ 
over a closed manifold $M$ with a smooth measure $\mu$. 
We lift all structures on $M$ to $\widetilde{M}$. 
Then $(\widetilde{M} , \widetilde{H})$ 
is a Heisenberg manifold, 
$\widetilde{P} : C^{\infty}(\widetilde{E}) \to C^{\infty}(\widetilde{F})$ is a 
Heisenberg elliptic differential operator 
and one has $\widetilde{Q}\widetilde{P} = 1 - \widetilde{S}_{0}$ 
and $\widetilde{P}\widetilde{Q} = 1 - \widetilde{S}_{1}$. 

Since $P$ is a differential operator (in particular, $P$ is local), 
there exists a constant $C = C(\widetilde{P} , \phi ) > 0$ such that 
we have an inequality 
\[
\| \widetilde{P}(\phi f) \|_{L^{2}} 
	\leq C( \|\chi \widetilde{P}f \|_{L^{2}} + \| \chi f \|_{L^{2}} ) 
\]
for any $f \in C^{\infty}(\widetilde{E})$; see \cite[Proposition 3.3.2]{MR2417549}. 
Here, 
$\phi , \chi \in C_{c}^{\infty}(\widetilde{M})$ are compactly supported 
smooth functions and 
one assumes $\chi = 1$ on the support of $\phi$. 
Thus by using Atiyah's technique of the proof of \cite[Proposition 3.1]{MR0420729}, 
we have the following: 

\begin{lem}
\label{lem:closure}
The minimal domain of $\widetilde{P}$ equals the maximal domain of $\widetilde{P}$. 
\end{lem}

By Lemma \ref{lem:closure}, 
$\widetilde{P}$ has the unique 
closed extension denoted by the same letter $\widetilde{P}$ 
and thus the closure of the formal adjoint of $\widetilde{P}$ 
(the formal adjoint is also Heisenberg elliptic) 
equals the Hilbert space adjoint $\widetilde{P}^{\ast}$. 

On the other hand, 
any $L^{2}$-solutions of $\widetilde{P}u = 0$ and $\widetilde{P}^{\ast}u = 0$ 
are smooth by 
existence of a parametrix. 
Thus the orthogonal projection $\Pi_{0}$ (resp. $\Pi_{1}$) 
onto a closed subspace 
$\mathrm{Ker}_{L^{2}}(\widetilde{P})$ (resp. $\mathrm{Ker}_{L^{2}}(\widetilde{P}^{\ast})$) 
of the $L^{2}$-sections 
is of $\Gamma$-trace class 
since operators $\phi \Pi_{0} \psi$ and $\phi \Pi_{1} \psi$ 
are smoothing operators on compact sets. 
Thus one obtains the well-defined $\Gamma$-index of $\widetilde{P}$. 

\begin{dfn}
The $\Gamma$-index of $\widetilde{P}$ is defined to be 
\[
\ind_{\Gamma}(\widetilde{P}) 
	= \mathrm{tr}_{\Gamma}(\Pi_{0}) - \mathrm{tr}_{\Gamma}(\Pi_{0}) \in \mathbb{R}. 
\]
\end{dfn}

By using operators $\widetilde{P}, \widetilde{Q}, \widetilde{S}_{0}$ 
and $\widetilde{S}_{1}$, we define  
\[ 
e_{\widetilde{P}} = 
\begin{bmatrix}
\widetilde{S_{0}}^{2} & 
	\widetilde{S_{0}}(1 + \widetilde{S_{0}}) \widetilde{Q} \\
\widetilde{S_{1}}\widetilde{P} & 1-\widetilde{S_{1}}^{2}
\end{bmatrix} \quad \text{and} \quad 
e_{1} = 
\begin{bmatrix}
0 & 0 \\ 
0 & 1
\end{bmatrix}. 
\]
Since a difference $e_{\widetilde{P}} - e_{1}$ is of $\Gamma$-trace class,  
one can define a $\Gamma$-index class of $\widetilde{P}$.  

\begin{dfn}
We define $\Gamma$-index class $\mathrm{Ind}_{\Gamma}(\widetilde{P})$ of $\widetilde{P}$ by 
\[
\mathrm{Ind}_{\Gamma}(\widetilde{P}) 
	= [e_{\widetilde{P}}] - [e_{1}] \in K_{0}(\mathcal{K}_{\Gamma}). 
\]
\end{dfn}

By using a $\Gamma$-trace, 
we have the $\Gamma$-index theorem 
in Heisenberg calculus. 

\begin{thm}
\label{thm:GammaHeisen}
Let $P$ be a Heisenberg elliptic differential operator 
on a closed Heisenberg manifold $(M,H)$ 
and $\widetilde{P}$ its lift as previously. 
Then one has 
\[
\ind_{\Gamma}(\widetilde{P}) 
	= (\mathrm{tr}_{\Gamma})_{\ast}(\mathrm{Ind}_{\Gamma}(\widetilde{P}))
	= \ind (P) \in \mathbb{Z}. 
\]
\end{thm}

\begin{proof}
First, note that equalities
\begin{align*}
1 - S_{0}^{2} &= 1 - (1 - QP)^{2} = (2Q - QPQ)P \quad \text{and} \\
1 - S_{1}^{2} &= 1 - (1 - PQ)^{2} = P(2Q - QPQ), 
\end{align*}
and note that operators $2Q - QPQ$ , $S_{0}^{2}$ and $S_{1}^{2}$ 
are almost local operators. 
Thus 
by Atiyah's technique in \cite[Section 5]{MR0420729}, 
one has 
\[
\ind_{\Gamma}(\widetilde{P}) 
	= \mathrm{tr}_{\Gamma}(\Pi_{0}) - \mathrm{tr}_{\Gamma}(\Pi_{1}) 
	= \mathrm{tr}_{\Gamma}(\widetilde{S_{0}}^{2}) - \mathrm{tr}_{\Gamma}(\widetilde{S_{1}}^{2}). 
\]

Next, the definition of a map $(\mathrm{tr}_{\Gamma})_{\ast}$, 
one has 
\[
(\mathrm{tr}_{\Gamma})_{\ast}(\mathrm{Ind}_{\Gamma}(\widetilde{P}))
	= \mathrm{tr}_{\Gamma}
		\begin{bmatrix}
				\widetilde{S_{0}}^{2} & 
				\widetilde{S_{0}}(1 + \widetilde{S_{0}}) \widetilde{Q} \\
				\widetilde{S_{1}}\widetilde{P} & -\widetilde{S_{1}}^{2}
				\end{bmatrix}
	= \mathrm{tr}_{\Gamma}(\widetilde{S_{0}}^{2}) - \mathrm{tr}_{\Gamma}(\widetilde{S_{1}}^{2}). 
\]

Since operators $\widetilde{S}_{0}^{2}$ and $\widetilde{S}_{1}^{2}$ are 
lifts of almost local smoothing operators, 
one has 
\[
\ind (P)
	= \mathrm{Tr}(S_{0}^{2}) - \mathrm{Tr}(S_{1}^{2}) 
	= \mathrm{tr}_{\Gamma}(\widetilde{S_{0}}^{2}) - \mathrm{tr}_{\Gamma}(\widetilde{S_{1}}^{2}) 
\]
by using (\ref{eq:gammatr}) in Section \ref{sec:atiyah}. 
This proves the equality in the theorem. 
\end{proof}

\begin{rmk}
As pointed out in \cite[Section 4]{Erppolycontant}, 
the results in \cite[Section 3]{MR2417549} hold 
verbatim for arbitrary codimension. 
That is, 
we do not need to assume that 
a distribution $H$ is of codimension $1$. 
\end{rmk}

\subsection{Example}
\label{sub:exm}

Index problems for Heisenberg elliptic operators 
on arbitrary closed Heisenberg manifolds 
are not solved yet. 
However, van Erp \cite{MR2680395,MR2680396} and 
Baum and van Erp \cite{MR3261009} solved 
the index problem on contact manifolds, 
which are 
good examples of Heisenberg manifolds. 
In this subsection, 
we investigate an example of Heisenberg elliptic differential operators 
with non-trivial $\Gamma$-index 
on a Galois covering over a closed contact manifold
by using the index formula for subLaplacians.  

Let $T^{2} = S^{1} \times S^{1} = \{ (e^{ix},e^{iy}) \}$ be a $2$-torus 
and set 
\[
e(x,y) = 
\begin{bmatrix}
f(x) & g(x) + h(x)e^{iy} \\
g(x) + h(x)e^{-iy} & 1-f(x)
\end{bmatrix}. 
\]
Here, let $f$ be a $[0,1]$-valued 
$2\pi \mathbb{Z}$-periodic function on $\mathbb{R}$ 
such that $f(0) = 1$ and $f(\pi) = 0$ 
and 
set $g(x) = \chi_{[0, \pi]}(x)\sqrt{f(x)-f(x)^{2}}$ and 
$h(x) = \chi_{[\pi, 2\pi]}(x)\sqrt{f(x)-f(x)^{2}}$; see \cite[Section I. 2]{Loringthesis}. 
Moreover, we assume $f$, $g$ and $h$ are smooth functions. 
Then $e$ defines an $M_{2}(\mathbb{C})$-valued 
smooth function on $T^{2}$. 

Since $e$ is an idempotent of rank $1$, 
$e$ defines an complex line bundle $E$ on $T^{2}$. 
As well known, 
the first Chern class $c_{1}(E)$ of $E$ is given by 
a $2$-form 
\[
\frac{-1}{2\pi i}\mathrm{tr}(e(de)^{2}) 
	= \frac{-1}{\pi}(hh' + 2f'h^{2} -2fhh')dx \wedge dy. 
\]
Thus, by using an equality $h^{2} = f-f^{2}$ on $[\pi , 2\pi]$, 
we can calculate the first Chern number of $E$: 
\[
\int_{T^{2}}c_{1}(E) = -\int_{\pi}^{2\pi}f'dx = -1. 
\]

Let $T^{3} = T^{2} \times S^{1} = \{ (e^{ix}, e^{iy} , e^{iz}) \}$ be a $3$-torus and 
$q : T^{3} \to T^{2}$ the projection onto $T^{2}$ of the first component. 
Set $\theta_{k} = \cos (kz)dx - \sin (kz)dy$ for a positive integer $k$, 
$H_{k} = \mathrm{ker}(\theta_{k})$, 
$f_{l}(x,y,z) = e^{ilz} + 1$ for a integer $l$ 
and $F = q^{\ast}E$. 
Then $(T^{3},H_{k})$ is a contact manifold 
and $H_{k}$ is a flat vector bundle. 
Denote by $T_{k}$ the Reeb vector field for $\theta_{k}$ and 
$\Delta_{H_{k}}^{F} 
	= -\nabla_{X_{k}}^{F}\nabla_{X_{k}}^{F} 
			- \nabla_{Y_{k}}^{F}\nabla_{Y_{k}}^{F}$ 
the sum of squares on $F$, 
where $\{ X_{k} , Y_{k} \}$ is a local frame of $H_{k}$.  
Set 
\[
P_{k,l} = \Delta_{H_{k}}^{F} + if_{l}\nabla_{T_{k}}^{F}. 
\]
Since the values of $f_{l} - n$ contained in $\mathbb{C}^{\times}$ for any odd integer $n$, 
an operator $P_{k,l} : C^{\infty}(F) \to C^{\infty}(F)$ 
is a Heisenberg elliptic differential operator of Heisenberg order $2$. 
By the index formula for $P_{k,l}$ in \cite[Example 6.5.3]{MR3261009}, one has 
\[
\ind (P_{k,l}) = \int_{T^{3}}\frac{-1}{2\pi i}e^{-ilz}de^{ilz} \wedge c_{1}(F) 
	= \frac{-1}{2\pi i}\int_{S^{1}}e^{-ilz}de^{ilz} \int_{T^{2}}c_{1}(E) 
	= l. 
\]

Note that 
a contact structure $H_{k}$ is a lift of $H_{1}$ by  
a $k$-fold cover $p_{k} : T^{3} \to T^{3}$; 
$(e^{ix},e^{iy},e^{iz}) \mapsto (e^{ix},e^{iy},e^{ikz})$. 
Since the lift $\widetilde{P_{1,l}}$ of a subLaplacian $P_{1,l}$ 
by $p_{k}$ equals $P_{k,kl}$, 
we have the $\Gamma (= \mathbb{Z}/k\mathbb{Z})$-index of $\widetilde{P_{1,l}}$: 
\[
\ind_{\Gamma}(\widetilde{P_{1,l}}) = \frac{1}{k}\ind (\widetilde{P_{1,l}})
	= \frac{1}{k}\ind (P_{k,kl}) = l = \ind (P_{1,l}). 
\]

Next, we consider a general Galois covering of $T^{3}$. 
Let $X \to T^{3}$ be a Galois covering 
with a deck transformation group $\Gamma$, 
which is a quotient of $\pi_{1}(T^{3}) = \mathbb{Z}^{3}$, 
for example, $X = \mathbb{R}^{3}$ and $\Gamma = \mathbb{Z}^{3}$ 
the universal covering. 
By Theorem \ref{thm:GammaHeisen}, 
we have a non-trivial $\Gamma$-index as follows: 
\[ 
\ind_{\Gamma} (\widetilde{P_{k,l}}) = \ind (P_{k,l}) = l. 
\]

\bibliographystyle{plain}
\bibliography{Gamma_Heisen_ref}

\begin{thebibliography}{1}

\bibitem{MR0420729}
M.~F. Atiyah.
\newblock Elliptic operators, discrete groups and von {N}eumann algebras.
\newblock In {\em Colloque ``{A}nalyse et {T}opologie'' en l'{H}onneur de
  {H}enri {C}artan ({O}rsay, 1974)}, pages 43--72. Ast\'erisque, No. 32--33.
  Soc. Math. France, Paris, 1976.

\bibitem{MR3261009}
P.~F. Baum and E.~van Erp.
\newblock {$K$}-homology and index theory on contact manifolds.
\newblock {\em Acta Math.}, 213(1):1--48, 2014.

\bibitem{MR1303779}
A.~Connes.
\newblock {\em Noncommutative geometry}.
\newblock Academic Press Inc., San Diego, CA, 1994.

\bibitem{Loringthesis}
T.~A. Loring.
\newblock {\em The Torus and Noncommutative Topology}.
\newblock PhD thesis, University of California, 1986.

\bibitem{MR2417549}
R.~S. Ponge.
\newblock Heisenberg calculus and spectral theory of hypoelliptic operators on
  {H}eisenberg manifolds.
\newblock {\em Mem. Amer. Math. Soc.}, 194(906):viii+ 134, 2008.

\bibitem{MR2680395}
E.~van Erp.
\newblock The {A}tiyah-{S}inger index formula for subelliptic operators on
  contact manifolds. {P}art {I}.
\newblock {\em Ann. of Math. (2)}, 171(3):1647--1681, 2010.

\bibitem{MR2680396}
E.~van Erp.
\newblock The {A}tiyah-{S}inger index formula for subelliptic operators on
  contact manifolds. {P}art {II}.
\newblock {\em Ann. of Math. (2)}, 171(3):1683--1706, 2010.

\bibitem{Erppolycontant}
E.~van Erp.
\newblock Contact structures of arbitrary codimension and idempotents in the
  heisenberg algebra, 2010.
\newblock arXiv:1001.5426.

\bibitem{MR2746652}
Erik van Erp.
\newblock The index of hypoelliptic operators on foliated manifolds.
\newblock {\em J. Noncommut. Geom.}, 5(1):107--124, 2011.

\end{thebibliography}

\end{document}